\documentclass[12pt,twoside=semi,leqno,abstracton
]{scrartcl}
\pdfoutput=1
\setkomafont{title}{\LARGE\bfseries\boldmath}
\setkomafont{section}{\large\bfseries\boldmath}
\setkomafont{subsection}{\large\bfseries\itshape\boldmath}
\setkomafont{subtitle}{\itshape\boldmath}
\setkomafont{disposition}{\normalfont\normalcolor\bfseries}
\setkomafont{dictum}{\scshape}
\setkomafont{paragraph}{\bfseries\boldmath}
\setkomafont{pagehead}{\normalfont\upshape\small}
\setkomafont{pagenumber}{\normalfont\upshape\normalsize}
\usepackage[arxiv%
]{MyKoma2}
\usepackage{mathtools}
\usepackage{tikz-cd}
\usetikzlibrary{babel}
\tikzcdset{%
  arrow style =math font}
\tikzcdset{every label/.append style = {font = \small}}
\makeatletter

\def\@InfoBox{%
  \if!\@keyAMSClassification!\else%
  \par\noindent{\textbf{MSC 2020: }%
    \@keyAMSClassification}\fi%
  \if!\@keyJELClassification!\else%
  \par\noindent{\textbf{JEL: }\@keyJELClassification}\fi%
  \if!\@Journal!\else\par\noindent\@Journal\fi%
  \if!\@DOI!\if!\@ArXiv!\else%
  \par{\href{http://arxiv.org/abs/\@ArXiv}{\normalcolor
      arXiv:\@ArXiv}}\fi%
  \else\par\noindent\@DOI\if!\@ArXiv!\else%
  \noindent{, \href{http://arxiv.org/abs/\@ArXiv}{\normalcolor
      arXiv:\@ArXiv}}\fi\fi%
  \if!\@history!\else%
  \def\iso@languagename{german}\dategerman\daymonthsepgerman{}
  \monthyearsepgerman{}{}%
  \par\@history\fi%
}

\definecolor{blendedblue}{rgb}{0.2,0.2,0.7}

\DeclareSymbolFont{symbolsC}{U}{pxsyc}{m}{n}
\DeclareMathSymbol{\medcirc}{\mathbin}{symbolsC}{7}
\DeclareMathSymbol{\medbullet}{\mathbin}{symbolsC}{8}

\usepackage{accents}

\DeclareFontFamily{U}{mathb}{\hyphenchar\font45}
\DeclareFontShape{U}{mathb}{m}{n}{
<-6> mathb5 <6-7> mathb6 <7-8> mathb7
<8-9> mathb8 <9-10> mathb9
<10-12> mathb10 <12-> mathb12
}{}
\DeclareSymbolFont{mathb}{U}{mathb}{m}{n}
\DeclareMathSymbol{\llcurly}{\mathrel}{mathb}{"CE}
\DeclareMathSymbol{\ggcurly}{\mathrel}{mathb}{"CF}

\DeclareMathOperator{\ca}{ca}

\makeatother

\begin{document}

\title{Semistatic robust utility indifference valuation and robust
  integral functionals}
\author{Keita Owari}

\address{Center for Advanced
  Research in Finance, and\\
  Graduate School of Economics,\\
  The University of Tokyo\newline 7-3-1 Hongo, Bunkyo-ku, Tokyo
  113-0033, JP }%
\thanks{This work is supported by CARF (Center for Advanced Research
  in Finance) in Graduate School of Economics of The University of
  Tokyo.}%
\email{owari@e.u-tokyo.ac.jp}

\runauthor{K. Owari}

\runtitle{Semistatic robust indifference valuation and integral functionals}

\CurrVer{\today}


\abstract{%
  We consider a discrete-time robust utility maximisation with
  \emph{\bfseries semistatic strategies}, and the associated
  indifference prices of exotic options.  For this purpose, we
  introduce a robust form of convex integral functionals on the space
  of bounded continuous functions on a Polish space, and establish
  some key regularity and representation results, in the spirit of the
  classical Rockafellar theorem, in terms of the duality formed with
  the space of Borel measures. These results (together with the
  standard Fenchel duality and minimax theorems) yield a duality for
  the robust utility maximisation problem as well as a representation
  of associated indifference prices, where the presence of static
  positions in the primal problem appears in the dual problem as a
  \emph{\bfseries marginal constraint} on the martingale
  measures. Consequently, the resulting indifference prices are
  consistent with the observed prices of vanilla options.}

\maketitle

\section{Introduction}
\label{sec:Intro}

This paper consists of two parts. The first part is concerned with the
following form of \emph{\bfseries robust convex integral functionals}:
\begin{equation}
  \label{eq:IntegFunctIntro}
  I_{\varphi,\PC}(f):=\sup_{P\in\PC}\int_\Omega \varphi(\omega,f(\omega))P(d\omega),
  \quad f\in C_b(\Omega),
\end{equation}
where $\Omega$ is a Polish space,
$\varphi:\Omega\times\RB\rightarrow\RB$ is a normal convex integrand,
and $\PC$ is a convex set of Borel probability measures on $\Omega$
that is compact for the weak topology induced by $C_b(\Omega)$. This
is a robust version of convex integral functionals. In the classical
case with $\PC$ being a singleton, say $\{\PB\}$, \citet{MR0310612}
considered (among many others) an integral functional
$I_{\varphi,\{\PB\}}=:I_{\varphi,\PB}$ on $L_\infty(\PB)$ and, under
mild integrability assumptions on $\varphi$, found its conjugate on
$L_\infty(\PB)'$ in the form:
\begin{equation*}
  \label{eq:RockafellarClassical}
  I_{\varphi,\PB}^*(\nu)=I_{\varphi^*,\PB}(d\nu_r/d\PB)
  +\sup_{\zeta\in L_\infty(\PB), I_{\varphi,\PB}(\zeta)<\infty}
  \nu_s(\zeta),
\end{equation*}
where $\nu=\nu_r+\nu_s$ is the (unique) Yosida-Hewitt decomposition of
$\nu\in L_\infty(\PB)'$ into regular ($\sigma$-additive) part $\nu_r$
and singular (purely finitely additive) part $\nu_s$, and
$\varphi^*(\omega,y)=\sup_{x\in\RB}(xy-\varphi(\omega,x))$.  In
particular, if $I_{\varphi,\PB}$ is finite everywhere on
$L_\infty(\PB)$, the second (singular) term is trivial, and
$I_{\varphi,\PB}$ is continuous for the Mackey topology
$\tau(L_\infty(\PB),L_1(\PB))$.  Similar representations on other
decomposable spaces of measurable functions (including
($\sigma$-finite) Orlicz spaces; e.g. \citep{MR577677}) are found as
well, and \citep{MR0310612} and \citep{MR3730431} obtained similar
results in the non-decomposable space $C_0(X)$.  For more information,
see \citep[Ch.~14]{MR1491362} and \citep{MR0512209}.

In Section~\ref{sec:RobInteg}, we establish a result in the spirit of
Rockafellar on the regularity and representation for the robust
integral functionals of the form (\ref{eq:IntegFunctIntro}) in terms
of the duality $\langle C_b(\Omega),\ca(\Omega)\rangle$, where
$\ca(\Omega)$ is the Banach space of (finite signed) Borel measures on
$\Omega$.  Specifically, under reasonable assumptions, we show that
$I_{\varphi,\PC}$ is continuous for the Mackey topology
$\tau(C_b(\Omega),\ca(\Omega))$, and that the conjugate, on
$\ca(\Omega)$, is a robust divergence functional associated to the
conjugate of $\varphi$ (Theorem~\ref{thm:Rockafellar1}). A similar
robust integral functional is considered in \citep{MR3400157}, but
there the set $\PC$ is supposed to be dominated by a single
probability measure $\PB$ and the domain space is $L_\infty(\PB)$
(which is decomposable), while we do not suppose $\PC$ is dominated,
and we chose $C_b(\Omega)$ (which is not decomposable) for the domain
space.

Our motivation for the robust integral functionals is to establish a
duality between a robust partial hedging and valuation problem for
exotic options \emph{\bfseries consistent with the observed prices of
  vanilla options}.  This is the content of the second part. In
Section~\ref{sec:RobUtil}, we consider a discrete-time robust utility
maximisation problem with semistatic strategies.  Basic ingredients
are the path-space $\RB^N$, the coordinate process
$(S_i)_{1\leq i\leq N}=\mathrm{id}_{\RB^N}$ with $S_0\equiv s_0$
(constant), and a sequence $(\mu_i)_{1\leq i\leq N}$ of distributions
on $\RB$ such that the set $\MC_\mu$ of martingale measures $Q$ for
$S$ with the marginal constraint
\begin{equation}
  \label{eq:MarginalsIntro}
  Q\circ S_i^{-1}=\mu_i,\quad i=1,...,N,
\end{equation}
is non-empty.  Each $Q\in\MC_\mu$ is thought of as a calibrated
pricing measure.  By a \emph{\bfseries semistatic} strategy, we mean a
pair $(H,(f_i)_{i\leq N})$ of predictable process $H=(H_i)_{i\leq N}$
and $(f_i)_{i\leq N}\in C_b(\RB)^N$, where each $f_i$ is viewed as a
vanilla option maturing at $i$ with payoff $f_i(S_i)$, which is
supposed to be priced in the market at
$\mu_i(f_i):=\int_\RB f_i d\mu_i$; so the gain from investing in
$(H,(f_i)_{i\leq N})$ is
\begin{align*}
  \sum\nolimits_{i\leq N}H_i(S_i-S_{i-1})+\sum\nolimits_{i\leq N}\left(f_i(S_i)-\mu_i(f_i)\right)
  =:H\bullet S_N+\Gamma_{(f_i)_{i\leq N}}.
\end{align*}
Then given a utility function $U:\RB\rightarrow\RB$, a possibly
non-dominated set $\PC$ of Borel probability measures on $\RB^N$, and
an exotic option specified by a real function $\Psi$ on the path-space
$\RB^N$, the basic robust utility maximisation problem is:
\begin{equation}
  \label{eq:RobUtilIntro}
  u_\Psi(x)=\sup_\xi\inf_{P\in\PC}\EB_P\left[U\left(x+\xi-\Psi\right)\right],\quad x\in\RB,
\end{equation}
where $\xi$ runs through (the gains from) suitable semi-static
strategies. (The precise formulation will be given in Section~\ref{sec:RobUtil}.)
This problem induces seller's and buyer's indifference prices of
$\Psi$; namely
\begin{align*}
  p_U^\mathrm{sell}(\Psi):=\inf\left\{x\in\RB: u_\Psi(x)\geq u_0(0)\right\};\quad
  p_U^\mathrm{buy}(\Psi)  =-p_U^\mathrm{sell}(-\Psi)\leq p_U^\mathrm{sell}(\Psi).
\end{align*}
As other indifference prices (see e.g. \citep{henderson09:_utilit});
$p_U^\mathrm{sell}(\Psi)$ is to be understood as the minimal price $p$
of $\Psi$ such that selling $\Psi$ at $p$ yields a better utility than
doing nothing.

By means of the regularity and representation results in
Section~\ref{sec:RobInteg}, we provide, in
Theorem~\ref{thm:DiscreteDuality}, a dual representation of the value
function $u_\Psi(x)$ where the dual problem is a minimisation over
$\MC_\mu$ of a certain robust divergence functional.  This duality
result yields a representation of the associated indifference prices
of $\Psi$:
\begin{align*}
  p_U^\mathrm{sell}(\Psi)=\sup_{Q\in\MC_\mu}\left(\EB_Q[\Psi]-\gamma_{V,\PC}(Q)\right),
\end{align*}
where $\gamma_{V,\PC}$ is a certain positive convex function on the
set of probability measures.  In particular,
$p_U^\mathrm{buy}(\Psi), p_U^\mathrm{sell}(\Psi)$ lie in the
\emph{\bfseries model-free pricing bound} in \citep{MR3066985}:
\begin{equation}
  \label{eq:BoundIntro}
  \Bigl[\inf_{Q\in\MC_\mu}\EB_Q[\Psi], \sup_{Q\in\MC_\mu}\EB_Q[\Psi]\Bigr],
\end{equation}
so the indifference prices are, in a certain sense, fair prices
consistent with calibration; see the last paragraph of this
introduction.

\runinhead{Related literature.} There is now a vast literature on
robust utility maximisation (without static positions), either
dominated or not; see \citep{MR4283309} for a survey with extensive
references.  To the best of our knowledge, the duality for the
semistatic robust utility maximisation problem (\ref{eq:RobUtilIntro})
with general utility function $U$ (on $\RB$), and all options
$f_i\in C_b(\RB)$, or (essentially) equivalently all call options
available for static positions at each $i$, is new, where the last
point appears in the dual problem as the \emph{\bfseries full exact
  marginal constraint}. However, \citep{MR3910012} (see also
\cite{MR4166747}) obtained a similar duality (with a different setup
of $\PC$) in the case of exponential utility with finitely many
vanilla options, in that (in our notation) each $f_i$ is restricted to
the span of a finite number of fixed options, say
$\Span(f_{i,k}; k\leq m_k)$, and accordingly, the constraint on
martingale measures in the dual problem is of a weaker form
$\EB_Q[f_{i,k}]=\mu_i(f_{i,k})$ instead of (\ref{eq:MarginalsIntro}).
Some numerical results are also presented in
\citep{pennanen2020optimal} to non-robust semistatic exponential
indifference valuation with finitely many options in illiquid market
(without duality).  Also, in \citep{guillaume:pastel-01002103}, a
robust exponential utility indifference valuation with a single
marginal constraint at the maturity (of exotic option) is considered
in a continuous time uncertain volatility framework.  There a utility
maximisation proof of Strassen's theorem (see (\ref{eq:Strassen1})
below) is also given.

Another related problem that originally inspired us is the
(multi-marginal) \emph{\bfseries martingale optimal transport (MOT)},
which is to minimise $\EB_Q[\Psi]$ over the set $\MC_\mu$. In this
line, \citep{MR3066985} proved that the infimum
$\inf_{Q\in\MC_\mu}\EB_Q[\Psi]$ is attained and is, in financial
terms, equal to the maximum sub-hedging cost for $\Psi$ by the
semistatic strategies (see \citep[Th.~1.1]{MR3066985} for the precise
statement). A similar duality holds for
$\sup_{Q\in\MC_\mu}\EB_Q[\Psi]$ as well with obvious changes.  (See
also \citep{MR3313756,MR3705784} for similar dualities in different
setups, and \citep{MR3456332,MR3706738} for recent developments of
(mainly 2-marginal) MOT.)  These duality results give the interval in
(\ref{eq:BoundIntro}) a clear financial meaning as the \emph{\bfseries
  model-free pricing bound consistent with calibration}. In
particular, our indifference prices can indeed be viewed as fair
prices consistent with calibration, and yield better (or not worse)
bounds at the cost of small hedging error; this was the original point
of view of this study though the quantitative evaluation as well as a
good choice of the set $\PC$ are left for further investigations.
There are also some nonlinear generalisations of MOT
(\citep{MR3959726}, \citep{MR4570376}), typically of the form
$\inf_{Q\in\MC} \left(\EB_Q[\Psi]+\sum_{k\leq N}\rho_k(Q\circ
  S^{-1}_k)\right)$ where $\rho_k$ is a convex penalty function on the
set of Borel probability measures on $\RB$ and $\MC$ is the set of all
martingale measures for $(S_i)_{i\leq N}$ (without marginal
constraint); the case with $\rho_k=\delta_{\{\mu_k\}}$ is the MOT.

\section{Robust Convex Integral Functionals}
\label{sec:RobInteg}

\subsection{Preliminaries}
\label{sec:IntegPrelim}

In this paper, all the vector spaces are real, and when $X$ is a
Banach space, $\BB_X$ denotes its closed unit ball.

In the sequel, $\Omega$ is a \emph{\bfseries Polish space} with the
Borel $\sigma$-field $\BC(\Omega)$, and $C_b(\Omega)$ is the Banach
space of bounded continuous function with
$\|f\|_\infty=\sup_{\omega\in \Omega}|f(\omega)|$, while $\ca(\Omega)$
is the Banach space of (finite signed) Borel measures on $\Omega$ with
$\|\nu\|=|\nu|(\Omega)$. Also, $L_p(\mu):=L_p(\Omega,\BC(\Omega),\mu)$
for positive $\mu\in\ca(\Omega)$, and $\mathrm{Prob}(\Omega)$ denotes
the closed convex subset of $\ca(\Omega)$ consisting of probability
measures (i.e. those $\mu\in\ca(\Omega)$ with $\mu\geq 0$ and
$\mu(\Omega)=1$). For $\mu\in\mathrm{Prob}(\Omega)$, we write
$\EB_\mu[f]$ for $\int_\Omega fd\mu$.  Next, if $\langle X,X'\rangle$
is a (separated) dual system, the weak topology $\sigma(X,X')$
(resp. the Mackey topology $\tau(X,X')$) is the weakest (resp. finest)
locally convex topology on $X$ consistent with the duality
$\langle X,X'\rangle$, i.e. making $X'$ the dual of $X$. More
concretely, $\sigma(X,X')$ is the topology of pointwise convergence on
$X'$ while $\tau(X,X')$ is the topology of uniform convergence on
$\sigma(X',X)$-compact absolutely convex subsets of $X'$; see
\citep[Sec.5.14-18]{MR2378491} or \citep[Ch.~2, Sec.~8-13]{MR0372565}
for details.

\runinhead{The duality $\langle C_b(\Omega),\ca(\Omega)\rangle$.}
\mbox{} The space $\ca(\Omega)$ is (isometrically isomorphic to) a
closed subspace of the dual $C_b(\Omega)'$ (proper unless $\Omega$ is
compact) via $\langle f,\mu\rangle=\int_\Omega fd\mu=:\mu(f)$, which
makes $\langle C_b(\Omega),\ca(\Omega)\rangle$ a (separated) dual
system.  In the sequel, we are basically interested in the duality
$\langle C_b(\Omega),\ca(\Omega)\rangle$, but the proof of
Theorem~\ref{thm:Rockafellar1} below also involves the duality
$\langle C_b(\Omega),C_b(\Omega)'\rangle$. The dual $C_b(\Omega)'$ is
identified as the space of (finite signed) regular Borel measures on
the Stone-\v{C}ech compactification $\beta\Omega$, and $\Omega$ is
(homeomorphic to) a dense $G_\delta$, hence Borel, subset of
$\beta\Omega$. Then $\ca(\Omega)$ is regarded as the subspace of
$C_b(\Omega)'$ consisting of those measures supported by (the image in
$\beta\Omega$ of) $\Omega$. The next lemma describes $\ca(\Omega)$ in
$C_b(\Omega)'$ sorely in terms of $\Omega$ (without passing to
$\beta\Omega$).
\begin{lemma}[\citep{MR2098271}, Prop.~5 on p.~IX.59 or
  \citep{MR2267655}, Th.~7.10.6]
  \label{lem:DualOfCb}
  An $F\in C_b(\Omega)'$ lies in $\ca(\Omega)$, i.e.
  $F(f)=\int_\Omega f d\mu$ for some $\mu\in\ca(\Omega)$ iff for any
  $\varepsilon>0$, there exists a compact set $K\subset \Omega$ such
  that $|F(g)|\leq \varepsilon$ whenever $g\in\BB_{C_b(\Omega)}$ and
  $g=0$ on $K$.
\end{lemma}

The following version of Prokhorov's theorem characterises the
$\sigma(\ca(\Omega),C_b(\Omega))$-compact sets; here the sufficiency
follows from Lemma~\ref{lem:DualOfCb} and
$\sigma(\ca(\Omega),C_b(\Omega))=\sigma(C_b(\Omega)',C_b(\Omega))|_{\ca(\Omega)}$
(so $\sigma(\ca(\Omega),C_b(\Omega))$-bounded $\Leftrightarrow$ weak*
bounded in $C_b(\Omega)'$ $\Leftrightarrow$ norm bounded), while the
necessity is by ``gliding hump''.
\begin{lemma}[Prokhorov's theorem; \citep{MR2267655}, Th.~8.6.7 and
  8.6.8]
  \label{lem:Prokhorov}\mbox{}
  (If $\Omega$ is Polish,) a set $\Lambda\subset\ca(\Omega)$ is
  relatively $\sigma(\ca(\Omega),C_b(\Omega))$-compact iff it is
  bounded in (total variation) norm and uniformly tight, i.e. for any
  $\varepsilon>0$, there exists a compact set $K\subset\Omega$ such
  that $\sup_{\nu\in\Lambda}|\nu|(\Omega\setminus K)<\varepsilon$.
\end{lemma}

\begin{remark}
  \label{rem:Prokhorov}
  The restriction of $\sigma(\ca(\Omega),C_b(\Omega))$ to the set of
  probability measures is precisely what probabilists call the weak
  topology, and Prokhorov's theorem is usually stated for probability
  measures, but we will need it on the whole $\ca(\Omega)$.
\end{remark}

\runinhead{Generalities on convex functions.} Given a duality
$\langle X,X'\rangle$, a convex function
$F:X\rightarrow\RB\cup\{+\infty\}$ is called proper if
$\dom(F):=\{x\in X:F(x)<\infty\}\neq \emptyset$. By the Hahn-Banach
theorem, such an $F$ is lower semicontinuous (lsc) for $\sigma(X,X')$,
equivalently for $\tau(X,X')$, iff $F$ has the Fenchel-Moreau dual
representation by $X'$:
\begin{align*}
  F(x)=\sup_{x'\in X'}\left(\langle x,x'\rangle-F^*(x')\right),\quad x\in X,
\end{align*}
where $F^*(x'):=\sup_{x\in X}(\langle x,x'\rangle-F(x))$, $x'\in X'$,
the conjugate of $F$. Also,
\begin{lemma}[Moreau-Rockafellar theorem; e.g. \citep{MR0160093}]
  \label{lem:MoreauRockafellar}
  Let $\langle X,X'\rangle$ be a dual pair. A finite lsc convex
  function $F:X\rightarrow\RB$ is $\tau(X,X')$-continuous at $0$ (then
  on the whole $X$) iff $F^*$ has $\sigma(X',X)$-compact sublevel
  sets, i.e. $\{x'\in X': F^*(x')\leq c\}$, $c\in\RB$, are
  $\sigma(X',X)$-compact. In this case,
  $F(x)=\max_{x'\in X'}\left(\langle x,x'\rangle-F^*(x')\right)$,
  $\forall x\in X$.
\end{lemma}

For the last part, note that $x'\mapsto F^*(x')-\langle x,x'\rangle$
is the conjugate of $y\mapsto F(x+y)$ which is $\tau(X,X')$-continuous
if $F$ is, so $\{x'\in X':\langle x,x'\rangle-F^*(x')\geq c\}$,
$c\in\RB$, are $\sigma(X',X)$-compact while
$\langle x,\cdot\rangle- F^*$ is upper semicontinuous for the same
topology.

\subsection{The functionals}
\label{sec:Ingredients}

Let $\PC$ be a set of Borel probability measures on $\Omega$ that we
suppose
\begin{equation}
  \label{eq:Tight}
  \PC\text{ is convex and $\sigma(\ca(\Omega),C_b(\Omega))$-compact,}
\end{equation}
but we \emph{\bfseries do not} suppose that it is dominated, i.e. that
there is a single probability measure $\PB$ such that $P\ll \PB$ for
all $P\in\PC$, so it does not embed in an $L_1$-space, and the common
uniform integrability arguments do not work.

Let $\LB_0(\PC)$ be the vector space of (equivalence classes modulo
equality ``$P$-a.s. for all $P\in\PC$'' of) real valued Borel
functions on $\Omega$, and define
\begin{align*}
  \LB_1(\PC)&:=\left\{\xi\in\LB_0(\PC):\|\xi\|_{1,\PC}:=\sup_{P\in\PC}\EB_P[|\xi|]<\infty\right\},\\
  \LB_{1,b}(\PC)&:=\left\{\xi\in \LB_1(\PC):\lim_n\|\xi\ind_{\{|\xi|>n\}}\|_{1,\PC}=0\right\}.
\end{align*}
It is known (see \citep{MR2754968}), and easily verified, that
$\LB_1(\PC)$ is a Banach space and $\LB_{1,b}(\PC)$ is its closed
subspace. In this paper, we just use these spaces to simplify the
notation. Also, as usual, we do not differentiate a class
$\xi\in \LB_0(\PC)$ and its representatives $f\in\xi$; so we regard
$C_b(\Omega)$ as a subspace of $\LB_0(\PC)$ consisting of those $\xi$
admitting a bounded continuous representative.
\begin{lemma}[\citep{MR2754968}, Prop.~19]
  \label{lem:UIL1B}
  If $\eta\in\LB_{1,b}(\PC)$, then any $\varepsilon>0$ admits a
  $\delta>0$ such that for any $A\in\BC(\Omega)$ with
  $\sup_{P\in\PC}(A)\leq \delta$, one has
  $\sup_{P\in\PC}\EB_P[|\eta|\ind_A]\leq\varepsilon$.  In particular
  (under (\ref{eq:Tight})), if $\eta\in\LB_{1,b}(\PC)$, we have
  \begin{align*}
    \forall\varepsilon>0,\,
    \exists\text{a compact }K\subset\Omega\text{ such that }    \sup_{P\in\PC}\EB_P[|\eta|\ind_{K^c}]\leq 
    \varepsilon.
  \end{align*}
\end{lemma}

The next ingredient is a random function
$\varphi:\Omega\times\RB\rightarrow\RB$ which we suppose:
\begin{assumption}
  \label{as:Integrand}
  $\varphi:\Omega\times\RB\rightarrow\RB$ is such that
  \begin{align}
    \label{eq:AssInteg1}
    &\forall \omega\in\Omega,\, x\mapsto \varphi(\omega,x)\text{ is an (everywhere finite) convex function};\\
    \label{eq:AssInteg2}
    &\forall x\in\RB,\, \omega\mapsto \varphi(\omega,x)\text{ is upper semicontinuous (usc)};\\
    \label{eq:AssInteg3}
    &\forall x\in\RB,\, \varphi(\cdot, x)^+\in\LB_{1,b}(\PC)\text{ i.e. }\lim_n\sup_{P\in\PC}\EB_P[\varphi(\cdot,x)\ind_{\{\varphi(\cdot, x)\geq n\}}]=0.\\
    \label{eq:AssInteg4}
&    \varphi(\cdot, 0)^-\in \LB_1(\PC).
  \end{align}
\end{assumption}

Several remarks on assumptions are in order. First, by
(\ref{eq:AssInteg1}), $x\mapsto \varphi(\omega,x)$,
$\omega\in \Omega$, are continuous (being a finite valued convex
function on $\RB$), while by (\ref{eq:AssInteg2}),
$\omega\mapsto \varphi(\omega,x)$, $x\in\RB$, are Borel; hence in the
terminology of convex analysis (see e.g. \citep{MR1491362}), $\varphi$
is a \emph{\bfseries convex Carathéodory integrand}, a fortiori it is
a (finite-valued) \emph{\bfseries normal convex integrand}, i.e.  the
\emph{epigraphical mapping}
\begin{align*}
  \omega\mapsto \{(x,\alpha)\in\RB\times\RB: \varphi(\omega,x)\leq \alpha\}
\end{align*}
is a closed-valued measurable multifunction. Then the (partial) conjugate
\begin{align*}
  \varphi^*(\omega,y):=\sup_{x\in\RB}(xy-f(\omega,x)),\quad\forall y\in\RB,
\end{align*}
is also a proper normal convex integrand, and Young's inequality
holds:
\begin{equation}
  \label{eq:Young1}
  xy\leq \varphi(\omega,x)+\varphi^*(\omega,y),\,\forall \omega\in\Omega,\,x,y\in\RB. 
\end{equation}
The normality of $\varphi$ implies that it is jointly measurable, so
$\omega\mapsto \varphi(\omega,f(\omega))$ is measurable whenever
$f:\Omega\rightarrow\RB$ is.  Then (\ref{eq:AssInteg2}) implies (hence
is equivalent to):
\begin{equation}
  \label{eq:AssInteg2Better}
  \tag{\text{\ref{eq:AssInteg2}${}^\prime$}}
  \forall  f\in C_b(\Omega),\, \omega\mapsto \varphi(\omega,f(\omega))\text{ is usc}.
\end{equation}
For if $f\in C_b(\Omega)$ and $\omega_n\rightarrow \omega$ in
$\Omega$, then $K=\{\omega, \omega_n;n\geq 1\}$ is compact, so by the
usc, $(\varphi(\omega',\cdot))_{\omega'\in K}$ is a pointwise bounded
family of convex functions; hence \citep[Th.~10.6]{MR0274683} gives us
a constant $c>0$ such that
$\sup_{\omega\in K}|\varphi(\omega,x)-\varphi(\omega,y)|\leq c|x-y|$
whenever $x,y\in [-\|f\|_\infty,\|f\|_\infty]$. Consequently,
\begin{align*}
  \limsup_n&\{\varphi(\omega_n,f(\omega_n))-\varphi(\omega,f(\omega))\}\\
           &\leq c\limsup_n|f(\omega_n)-f(\omega)|
    +\limsup_n\varphi(\omega_n,f(\omega))-\varphi(\omega,f(\omega))\leq 0.
\end{align*}
Similarly, (\ref{eq:AssInteg3}) already implies
\begin{equation}
  \label{eq:AssInteg3Better}
  \tag{\text{\ref{eq:AssInteg3}${}^\prime$}}
  \forall f\in C_b(\Omega), \, \varphi(\cdot, f)^+\in \LB_{1,b}(\PC).
\end{equation}
Indeed, by convexity,
$\varphi(\cdot,f)\leq
\varphi(\cdot,-\|f\|_\infty)^++\varphi(\cdot,\|f\|_\infty)^+\in
\LB_{1,b}(\PC)$.

Finally, given (\ref{eq:AssInteg1}) and (\ref{eq:AssInteg3}),
(\ref{eq:AssInteg4}) is equivalent to:
\begin{equation}
  \label{eq:IntegrandInteg2}
  \exists \eta\in\LB_1(\PC)\text{ such that }\varphi^*(\cdot, \eta)^+\in\LB_1(\PC).
\end{equation}
Indeed, by the normality of $\varphi^*$ and
$\varphi(\cdot, 0)=\sup_y( -\varphi^*(\cdot, y))$,
$\Xi(\omega):=\{y\in\RB: \varphi(\omega, 0)\leq
1-\varphi^*(\omega,y)\}$ is a nonempty closed valued measurable
multifunction.  Thus Kuratowski-Ryll-Nardzewski's measurable selection
theorem yields a measurable function $f:\Omega\rightarrow\RB$ such
that $f(\omega)\in \Xi(\omega)$, i.e.
$\varphi(\omega,0)\leq 1-\varphi^*(\omega,f(\omega))$. Thus
$\varphi^*(\cdot, f)\leq1-\varphi(\cdot,0)^-\in \LB_1(\PC)$, and
$|f|=f\mathrm{sgn}(f)\leq \varphi(\cdot, -1)^++\varphi(\cdot,
1)^++\varphi^*(\cdot,f)^+\in \LB_1(\PC)$.

Now we define
\begin{equation}
  \label{eq:IntegFunctDef}
  I_{\varphi,\PC}(f):=\sup_{P\in\PC}\EB_P[\varphi(\cdot,f)]
  =\sup_{P\in\PC}\int_\Omega\varphi(\omega,f(\omega))P(d\omega),\quad f\in C_b(\Omega).
\end{equation}
We check that this is well-defined under (\ref{eq:Tight}) and
Assumption~\ref{as:Integrand}.
\begin{lemma}
  \label{lem:IntegFuncWellDefined}
  Under (\ref{eq:Tight}) and Assumption~\ref{as:Integrand},
  $I_{\varphi,\PC}$ is well-defined as a finite-valued convex function
  on $C_b(\Omega)$, and it is lower semicontinuous for the topology of
  pointwise convergence on bounded sets:
  \begin{equation}
    \label{eq:IntegFuncQSLSC}
    \sup_n\|f_n\|_{\infty}<\infty,\,f_n\rightarrow f\text{ pointwise }\Rightarrow\,
    I_{\varphi,\PC}(f)\leq\liminf_nI_{\varphi,\PC}(f_n).
  \end{equation}
  In particular, $I_{\varphi,\PC}$ is norm-lsc, hence norm continuous
  (being finite-valued) on $C_b(\Omega)$.
\end{lemma}
\begin{proof}
  Picking an $\eta\in\LB_1(\PC)$ as in (\ref{eq:IntegrandInteg2}),
\begin{align*}
  \forall f\in C_b(\Omega),\,
  \varphi(\cdot,f)\geq -\|f\|_{\infty}|\eta|-\varphi^*(\cdot,\eta)^+\in\LB_1\subset\mcap_{P\in\PC}L_1(P).
\end{align*}
Thus for each $P\in\PC$, $f\mapsto \EB_P[\varphi(\cdot,f)]$ is
well-defined as a convex function on $C_b(\Omega)$, and it is lsc on
bounded sets for the pointwise convergence by Fatou's lemma, hence so
is their pointwise supremum
$I_{\varphi,\PC}(f)=\sup_{P\in\PC}\EB_P[\varphi(\cdot,f)]$.  Then
(\ref{eq:AssInteg3Better}) guarantees that
$I_{\varphi,\PC}(f)<\infty$ for all $f\in C_b(\Omega)$.
\end{proof}

We next define the $\varphi^*$-divergence functional by
\begin{align}
  \label{eq:Entropy1}
  J_{\varphi^*}(\nu|P)&:=\ccases{
                        \EB_P[\varphi^*(\cdot, d\nu/dP)]&\text{if }\nu\ll P,\\
  +\infty&\text{otherwise},}\quad\forall \nu\in \ca(\Omega),\, P\in\PC.
\end{align}
This is a jointly convex function with values in $(-\infty,\infty]$.
To see this, let
\begin{align*}
  \tilde\varphi^*(\cdot, y,z):=\sup_x(xy-z\varphi(\cdot, x))
  =z\varphi^*\left(\cdot,\frac{y}{z}\right)\ind_{\{z>0\}}
  +\infty\ind_{\{y\neq 0, z=0\}}\geq - z\varphi(\cdot,0)^+.
\end{align*}
This $\tilde\varphi^*$ is convex in $(y,z)\in \RB\times\RB^+$ and
$J_V(\nu|P)=\EB_\PB[\tilde\varphi^*(\cdot,
\frac{d\nu}{d\PB},\frac{dP}{d\PB})]>-\infty$ (by (\ref{eq:AssInteg3}))
whenever $\nu,P\ll\PB\in\mathrm{Prob}(\Omega)$. Since any finite number
of finite signed measures are dominated by a single probability, $J_V$
is jointly convex on $\ca(\Omega)\times\PC$, and $J_V(\nu|P)<\infty$
if $d\nu/dP=\eta\in \LB_1(\PC)$ as in (\ref{eq:IntegrandInteg2}).
Since $\PC$ is convex,
\begin{equation}
  \label{eq:RobDiv1}
  J_{\varphi^*,\PC}(\nu):=\inf_{P\in\PC}J_{\varphi^*}(\nu|P),\quad\forall \nu\in \mathrm{ca},
\end{equation}
called \emph{\bfseries the robust $\varphi^*$-divergence functional},
is convex on $\ca(\Omega)$ and not identically $+\infty$. Further,
(\ref{eq:Young1}) yields
$\nu(f)\leq \EB_P[\varphi(\cdot, f)]+J_{\varphi^*}(\nu|P)\leq
I_{\varphi,\PC}(f)+J_{\varphi^*}(\nu|P)$ for any $f\in C_b(\Omega)$,
$\nu\in \ca(\Omega)$, $P\in\PC$, and taking the infimum over
$P\in\PC$,
\begin{equation}
  \label{eq:YoungRockafellar1}
  \nu(f)\leq  I_{\varphi,\PC}(f)+J_{\varphi^*,\PC}(\nu),\quad\forall f\in C_b(\Omega),\,\nu\in\ca(\Omega).
\end{equation}
In particular,
$J_{\varphi^*,\PC}(\nu)\geq -I_{\varphi,\PC}(0)>-\infty$, so
$J_{\varphi^*,\PC}$ is proper.

\subsection{A Duality Result}
\label{sec:MainInteg}

Now the main result of this section is the following:
\begin{theorem}[Rockafellar-Type Duality]
  \label{thm:Rockafellar1}
  Suppose (\ref{eq:Tight}), and Assumption~\ref{as:Integrand}. Then
  $I_{\varphi,\PC}:C_b(\Omega)\rightarrow\RB$ is
  $\tau(C_b(\Omega),\ca(\Omega))$-continuous, and its conjugate is
  given by
  \begin{equation}
    \label{eq:ConjCB1}
    I_{\varphi,\PC}^*(\nu)
    := \sup_{f\in C_b(\Omega)}\left(\nu(f)-I_{\varphi,\PC}(f)\right)=J_{\varphi^*,\PC}(\nu),
    \quad \nu\in \ca(\Omega).
  \end{equation}
  In particular, it holds that
  \begin{equation}
    \label{eq:DualRepMax}
    I_{\varphi,\PC}(f)=\max_{\nu\in \ca(\Omega)}\left(\nu(f)-J_{\varphi^*,\PC}(\nu)\right),\quad f\in C_b(\Omega).
  \end{equation}
\end{theorem}

In view of the Moreau-Rockafellar theorem
(Lemma~\ref{lem:MoreauRockafellar}), the
$\tau(C_b(\Omega),\ca(\Omega))$-continuity is equivalent to (1)
$I_{\varphi,\PC}$ is $\sigma(C_b(\Omega),\ca(\Omega))$-lsc, and (2)
the conjugate $I_{\varphi,\PC}^*$ has
$\sigma(\ca(\Omega),C_b(\Omega))$-compact sublevels, i.e.
\begin{equation}
  \label{eq:InfCompact1}
  \forall c\in\RB,\,  \Lambda_c:=\{\nu\in\ca(\Omega): I_{\varphi,\PC}^*(\nu)\leq c\}
  \text{ is $\sigma(\ca(\Omega),C_b(\Omega))$-compact}.
\end{equation}
Regarding (1), we already know from
Lemma~\ref{lem:IntegFuncWellDefined} that $I_{\varphi,\PC}$ is
\emph{\bfseries sequentially} lsc for
$\sigma(C_b(\Omega),\ca(\Omega))$ (the
$\sigma(C_b(\Omega),\ca(\Omega))$-convergence implies the pointwise
convergence as $\ca(\Omega)$ contains the point masses), but it need
not imply the full lower semicontinuity (i.e. for nets, not only for
sequences). Indeed, unlike $L_\infty(\PB)$ which is a common choice of
the domain space in the dominated case, $C_b(\Omega)$ is neither a
dual (unless $\beta\Omega$ is extremally disconnected) nor a predual
space (unless $\Omega$ is compact). Thus the common techniques using
Krein-Šmulian theorem as well as a probabilistic description of the
Mackey topology on bounded sets (due to Grothendieck) are not
available.

\begin{remark}[Infinite $\varphi$?]
  \label{rem:InfinitePhi}
  \mbox{} Many of existing results on (classical) convex integral
  functionals are stated for possibly infinite proper integrand
  $\varphi$, where an additional singular term appears in the the
  conjugate as (\ref{eq:IntegFunctIntro}). In the robust case,
  however, our previous work \citep{MR3400157} suggests that this type
  of ``exact'' representation does not generally hold when $\varphi$
  is infinite (even if $\PC$ is dominated), and the current situation
  is even more complicated due to the non-decomposability of
  $C_b(\Omega)$.  Anyway, with infinite $\varphi$, we can no longer
  hope for the $\tau(C_b(\Omega),\ca(\Omega))$-continuity of the
  integral functional, and we are forced to work with the duality
  $\langle C_b(\Omega), C_b(\Omega)'\rangle$ which is not what we want
  in view of financial application.  We thus do not seek this
  direction in this paper.
\end{remark}

Given the $\tau(C_b(\Omega),\ca(\Omega))$-continuity (on the whole
$C_b(\Omega)$), the Fenchel duality theorem (see
e.g. \citep{MR743904}, Th.7.15 with $g=\delta_\CC$) yields that
\begin{corollary}[Meta duality]
  \label{cor:MetaDuality}
  Under the assumptions of Theorem~\ref{thm:Rockafellar1}, it holds
  for any nonempty convex set $\CC\subset C_b(\Omega)$ that
  \begin{equation}
    \label{eq:Metaduality1}
    \inf_{f\in \CC}I_{\varphi,\PC}(f)
    =-\min_{\nu\in \ca(\Omega)}\Bigl(J_{\varphi^*,\PC}(-\nu)
      +\sup_{g\in\CC}\nu(g)\Bigr).
    \end{equation}
    If in addition $\CC$ is a convex cone, the RHS is equal to
    $-\min_{\nu\in\CC^\medcirc}J_{\varphi^*,\PC}(-\nu)$, where
    $\CC^\medcirc$ is the one-sided polar of $\CC$ in
    $\langle C_b(\Omega),\ca(\Omega)\rangle$, i.e.
    \begin{align*}
      \CC^\medcirc=\{\nu\in\ca(\Omega): \nu(g)\leq 1,\,\forall g\in \CC\}.
    \end{align*}
\end{corollary}

A typical and motivating example of normal integrand $\varphi$ is of the following type:

\begin{proposition}
  \label{prop:GoodIntegrand}
  Let $\varphi:\RB\rightarrow\RB$ be a (deterministic finite-valued)
  convex function, and $B:\Omega\rightarrow\RB$ be a usc function such
  that
  \begin{equation}
    \label{eq:PropGoodIntegAss1}
    \varphi((1+\varepsilon)B)^+\in \LB_{1,b}(\PC)
    \text{ and }\varphi(-\varepsilon B)^+\in \LB_1(\PC)\text{ for some }\varepsilon>0.
  \end{equation}
  Then $\varphi_B(\omega,x):=\varphi(x+B(\omega))$ satisfies
  Assumption~\ref{as:Integrand} with the conjugate
  \begin{align*}
    \varphi_B^*(\cdot, y)=\sup_x\left(xy-\varphi(x+B)\right)
    =\varphi^*(y)-yB.
  \end{align*}
  Thus under (\ref{eq:Tight}), the functional $I_{\varphi_B,\PC}$ is
  $\tau(C_b(\Omega),\ca(\Omega))$-continuous on $C_b(\Omega)$ with
  conjugate $J_{\varphi^*_B,\PC}$. Further,
  $J_{\varphi^*,\PC}(\nu)<\infty$ $\Leftrightarrow$
  $J_{\varphi^*_B,\PC}(\nu)<\infty$, and $B\in L_1(\nu)$ whenever
  $\in_{\lambda>0}J_{V,\PC}(\lambda\nu)<\infty$. In particular,
  \begin{equation}
    \label{eq:GoodInetgConj1}
    J_{\varphi^*_B,\PC}(\nu)
    =\ccases{
      J_{\varphi^*,\PC}(\nu)-\nu(B)&\text{ if }\inf_{\lambda>0}J_{\varphi^*,\PC}(\lambda\nu)<\infty,\\
      +\infty&\text{ otherwise.}}
  \end{equation}
\end{proposition}

\begin{proof}
  The deterministic convex function $\varphi$ clearly satisfies
  Assumption~\ref{as:Integrand}, and $\varphi_B$ is usc in $\omega$
  since $B$ is. Thus $\varphi_B$ is a Carathéodory integrand.
  Denoting $\rho_\alpha(x)=\frac1\alpha\varphi(\alpha x)^+$,
  $\alpha>0$, the convexity of $\varphi$ yields
  \begin{equation}
    \label{eq:ExEstim1}
    \frac{1+\varepsilon}\varepsilon \varphi\left(\frac\varepsilon{1+\varepsilon}x\right)
    -\rho_\varepsilon(- B)
    \leq\varphi_B(x)
    \leq\frac{\varepsilon}{1+\varepsilon}\varphi\left(\frac{1+\varepsilon}\varepsilon x\right)
    +\rho_{1+\varepsilon}(B).
  \end{equation}
  Thus (\ref{eq:PropGoodIntegAss1}) shows that $\varphi_B$ satisfies
  (\ref{eq:AssInteg3}) and (\ref{eq:AssInteg4}). Further, taking the
  conjugate,
  $\varphi_B^*(\cdot, y)=\sup_x\left(xy-\varphi(x+B)\right)
  =\varphi^*(y)-yB$, and
\begin{equation}
  \label{eq:ExEstimConj1}
  \begin{split}
    \frac\varepsilon{1+\varepsilon}\varphi^*(y)-\rho_{1+\varepsilon}(B)
    &\leq\varphi_B^*(y)
      \leq \frac{1+\varepsilon}\varepsilon\varphi^*(y)+\rho_\varepsilon(- B).
  \end{split}
\end{equation}
This shows that $J_{\varphi^*_B,\PC}(\nu)<\infty$ $\Leftrightarrow$
$J_{\varphi^*,\PC}(\nu)<\infty$ $\Rightarrow$ $B\in L_1(\nu)$. Then
noting that $B\in L_1(\nu)$ iff $B\in L_1(\lambda\nu)$ for some (then
any) $\lambda>0$, we see that $B\in L_1(\nu)$ whenever
$\inf_{\lambda>0}J_{V,\PC}(\lambda\nu)<\infty$; in particular,
(\ref{eq:GoodInetgConj1}) holds.
\end{proof}

\subsection{Proof of Theorem~\ref{thm:Rockafellar1}}
\label{sec:ProofIntFunct}

Though Theorem~\ref{thm:Rockafellar1} is stated entirely in terms of
the duality $\langle C_b(\Omega),\ca(\Omega)\rangle$ (with
$\ca(\Omega)$ rather than $C_b(\Omega)'$), the proof relies on the
duality $\langle C_b(\Omega),C_b(\Omega)'\rangle$.  By
Lemma~\ref{lem:IntegFuncWellDefined}, $I_{\varphi,\PC}$ is norm
continuous. Since the norm topology is the Mackey topology
$\tau(C_b(\Omega),C_b(\Omega)')$, Moreau-Rockafellar's theorem
(Lemma~\ref{lem:MoreauRockafellar}) tells us that
\begin{equation}
  \label{eq:IntegFunctConjCDual}
  I_{\varphi,\PC}(f)=\sup_{\nu\in C_b(\Omega)'}\left(\nu(f)-I_{\varphi,\PC}^*(\nu)\right),\quad f\in C_b(\Omega),
\end{equation}
where the conjugate
$I_{\varphi,\PC}^*(\nu)=\sup_{f\in
  C_b(\Omega)}(\nu(f)-I_{\varphi,\PC}(f))$ is now considered on
$C_b(\Omega)'$, and
$\{\nu\in C_b(\Omega)': I_{\varphi,\PC}(\nu)\leq c\}$, $c\in\RB$, are
$\sigma(C_b(\Omega)',C_b(\Omega))$-compact ($\Leftrightarrow$ closed
and bounded in norm). All we need to get the
$\tau(C_b(\Omega),\ca(\Omega))$-continuity of $I_{\varphi,\PC}$ is to
replace the dual $C_b(\Omega)'$ by $\ca(\Omega)$.

\begin{proof}[Proof of Theorem~\ref{thm:Rockafellar1}: Mackey continuity]
  We first claim that
  \begin{equation}
    \label{eq:ConjEliminteSingular}
    I_{\varphi,\PC}^*(\nu)=\infty\text{ if }\nu\in C_b(\Omega)'\setminus \ca(\Omega).
  \end{equation}
  To see this, note first that for any $\nu\in C_b(\Omega)'$,
    \begin{equation}
    \label{eq:ProofConjInfiniteIrregular1}
    \sup_{f\in C_b(\Omega)}\left(\nu(f)-I_{\varphi,\PC}(f)\right)
    \geq\sup_n \sup_{g\in \BB_{C_b(\Omega)}}\left(\nu (ng )-I_{\varphi,\PC}(n g)\right),
  \end{equation}
  while for any $g\in \BB_{C_b(\Omega)}$,
  $\varphi(\cdot, ng)^+\leq \varphi(\cdot, -n)^+ +\varphi(\cdot,
  n)^+=:\beta_n\in \LB_{1,b}(\PC)$ by convexity and
  (\ref{eq:AssInteg3}). Thus Lemma~\ref{lem:UIL1B} yields compact sets
  $K_n\subset \Omega$, $n\geq 1$, with
  \begin{align*}
    \sup_{g\in \BB_{C_b(\Omega)}}\sup_{P\in\PC}\EB_P[\varphi(\cdot, ng)^+\ind_{K_n^c}]
    \leq\sup_{P\in\PC}\EB_P[\beta_n\ind_{K_n^c}]\leq 1.
  \end{align*}
  On the other hand, if $\nu\in C_b(\Omega)'\setminus C_b(\Omega)$,
  Lemma~\ref{lem:DualOfCb} gives an $\varepsilon>0$ and a sequence
  $g_n\in\BB_{C_b(\Omega)}$ such that $g_n\ind_{K_n}=0$ and
  $\nu(g_n)>\varepsilon$ (since $|\nu(g_n)|=\nu(g_n)\vee \nu(-g_n)$);
  hence
  \begin{align*}
    \nu( ng_n)-I_{\varphi,\PC}(n g_n)
    &\geq n\varepsilon -\underbrace{\sup_{P\in\PC}\EB_P[\varphi(\cdot, ng_n)^+\ind_{K^c_n}]}_{\leq 1}
      -\underbrace{\sup_{P\in\PC}
      \EB_P[\varphi(\cdot, 0)^+]}_{<\infty}.
  \end{align*}
  Combined with (\ref{eq:ProofConjInfiniteIrregular1}), we deduce
  $\sup_{f\in C_b(\Omega)}(\nu(f)-I_{\varphi,\PC}(f))=\infty$.

  Now by (\ref{eq:ConjEliminteSingular}) and
  (\ref{eq:IntegFunctConjCDual}), we have
  \begin{align*}
    I_{\varphi,\PC}(f)
    \stackrel{\text{(\ref{eq:IntegFunctConjCDual})}}=
    \sup_{\nu\in C_b(\Omega)'}\left(\nu(f)-I_{\varphi,\PC}^*(\nu)\right)
    \stackrel{\text{(\ref{eq:ConjEliminteSingular})}}=
    \sup_{\nu\in \ca(\Omega)}\left(\nu(f)-I_{\varphi,\PC}^*(\nu)\right).
  \end{align*}
  Thus $I_{\varphi,\PC}$ is
  $\sigma(C_b(\Omega),\ca(\Omega))$-lsc. (\ref{eq:ConjEliminteSingular})
  shows also that the sublevel set $\Lambda_c$ in $\ca(\Omega)$
  coincides with that considered in $C_b(\Omega)'$, i.e.
  \begin{align*}
    \Lambda_c=\left\{\nu\in\ca(\Omega):I_{\varphi,\PC}^*(\nu)\leq c\right\}
    =\left\{\nu\in C_b(\Omega)':I_{\varphi,\PC}^*(\nu)\leq c\right\}.
  \end{align*}
  As noted above (see the comment following
  (\ref{eq:IntegFunctConjCDual})), the last set is
  $\sigma(C_b(\Omega)',C_b(\Omega))$-compact. Consequently,
  $\Lambda_c$ is a $\sigma(C_b(\Omega)',C_b(\Omega))$-compact subset
  of $C_b(\Omega)'$ lying in $\ca(\Omega)$, so it is
  $\sigma(C_b(\Omega)',C_b(\Omega))|_{\ca(\Omega)}
  =\sigma(\ca(\Omega),C_b(\Omega))$-compact. Now being
  $\sigma(C_b(\Omega),\ca(\Omega))$-lsc with the conjugate having
  $\sigma(\ca(\Omega),C_b(\Omega))$-compact sublevels,
  Lemma~\ref{lem:MoreauRockafellar} shows that $I_{\varphi,\PC}$ is
  $\tau(C_b(\Omega),\ca(\Omega))$-continuous.
\end{proof}

We proceed to the conjugate formula (\ref{eq:ConjCB1}). We derive it
from the classical Rockafellar theorem on $L_\infty(\PB)$ and a
minimax argument. The latter needs the following simple lemma.
\begin{lemma}
  \label{lem:LSCP}
  Suppose (\ref{eq:Tight})--(\ref{eq:AssInteg4}). Then for each
  $f\in C_b(\Omega)$, $P\mapsto \EB_P[\varphi(\cdot, f)]$ is (affine,
  hence concave and) $\sigma(\ca(\Omega),C_b(\Omega))$-usc on $\PC$.
\end{lemma}
\begin{proof}
  Let $g:=\varphi(\cdot, f)$ which is usc with
  $g^+\in \mcap_{P\in\PC}L_1(P)$, and $g_m:=g\vee (-m)$.  Since
  $\EB_P[g]=\inf_m\EB_P[g_m]$, it suffices that $P\mapsto \EB_P[g_m]$,
  $m\geq 1$, are usc.  For each $n$, $g_m\wedge n$ is a bounded usc
  function, so $P\mapsto \EB_P[g_m\wedge n]$ is
  $\sigma(\ca(\Omega),C_b(\Omega))$-usc on
  $\PC\subset\mathrm{Prob}(\Omega)$ (see
  e.g. \citep[Th.~15.5]{MR2378491}).  Then note that
  $g_m-g_m\wedge n\leq g_m\ind_{\{g_m>n\}}=g^+\ind_{\{g>n\}}$, so
  $\lim_n\sup_{P\in \PC}\EB_P\left[g_m-g_m\wedge n\right]=0$ by
  (\ref{eq:AssInteg3Better}) ($\Leftarrow$ (\ref{eq:AssInteg3})). Thus
  if $P_k\rightarrow P$ in $(\PC,\sigma(\ca(\Omega),C_b(\Omega))$, one
  has
  \begin{align*}
    \limsup_k    \EB_{P_k}[g_m]
    &\leq \sup_{P\in\PC}\EB_{P}[g_m-g_m\wedge n]
      +\underbrace{\limsup_k\EB_{P_k}[g_m\wedge n]}_{\leq \EB_P[g_m\wedge n]\leq\EB_P[g_m]}.
  \end{align*}
  Letting $n\rightarrow\infty$, we get
  $\limsup_k\EB_{P_k}[g_m]\leq \EB_P[g_m]$. Since
  $\sigma(\ca(\Omega),C_b(\Omega))$ is metrisable on
  $\PC\subset\mathrm{Prob}(\Omega)$, this proves the claim.
\end{proof}

\begin{proof}[Proof of Theorem~\ref{thm:Rockafellar1}: the conjugate
  formula (\ref{eq:ConjCB1})]
  \mbox{} Given $\nu\in \ca(\Omega)$, the function
  $(f,P)\mapsto \nu(f)-\EB_P[\varphi(\cdot,f)]$ on
  $C_b(\Omega)\times\PC$, is concave in $f\in C_b(\Omega)$, and
  $\sigma(\ca(\Omega),C_b(\Omega))$-lsc and convex in $P\in\PC$ by
  Lemma~\ref{lem:LSCP}. Since the set $\PC$ is
  $\sigma(\ca(\Omega),C_b(\Omega))$-compact, the (usual) minimax
  theorem yields that
  \begin{align*}
    \sup_{f\in C_b(\Omega)}\left(\nu(f)-I_{\varphi,\PC}(f)\right)
    &=\sup_{f\in C_b(\Omega)}\inf_{P\in\PC}\left(\nu(f)-\EB_P[\varphi(\cdot, f)]\right)\\
    &=\inf_{P\in\PC}\sup_{f\in C_b(\Omega)}\left(\nu(f)-\EB_P[\varphi(\cdot, f)]\right).
  \end{align*}
  Therefore it suffices to show that
  \begin{equation}
    \label{eq:RockafellarProofFinal1}
    \forall \nu\in\ca(\Omega),\,\forall P\in\PC,\,
    \sup_{f\in C_b(\Omega)}\left(\nu(f)-\EB_P[\varphi(\cdot, f)]\right)=J_\varphi(\nu|P).
  \end{equation}
  So fix $\nu\in\ca(\Omega)$, $P\in\PC$, and pick a probability
  measure $\PB$ on $(\Omega,\BC(\Omega))$ with $\nu,P\ll\PB$
  (e.g. $\PB=\frac12(|\nu|/\|\nu\|+P)$). Then consider
  \begin{align*}
    \varphi_P(\omega,x):=\frac{dP}{d\PB}\varphi(\omega,x).
  \end{align*}
  This is a finite-valued normal convex integrand with the conjugate
  \begin{align*}
    \varphi_P^*(\cdot, y)=\frac{dP}{d\PB}\varphi^*(\cdot, y/(dP/d\PB))\ind_{\{dP/d\PB>0\}}
    +\infty\ind_{\{dP/d\PB=0, y\neq 0\}}.
  \end{align*}
  Note that $\varphi_P(\cdot, x)^+\in L_1(\PB)$ for all $x\in\RB$ by
  (\ref{eq:AssInteg3}), and by (\ref{eq:IntegrandInteg2}),
  $\varphi^*_P(\cdot, \zeta)^+\in L_1(\PB)$ for some
  $\zeta\in L_1(\PB)$ (with a slight abuse of notation,
  $\zeta=\eta \frac{dP}{d\PB}$, with the $\eta$ in
  (\ref{eq:IntegrandInteg2}) does the job). Thus the classical
  Rockafellar theorem (\cite[Th.1]{MR0310612}) shows that
  \begin{align*}
    \sup_{\xi\in L_\infty(\PB)}\left(\nu(\xi)-\EB_\PB[\varphi_P(\cdot, \xi)]\right)
    =\EB_\PB\left[\varphi^*_P(\cdot, d\nu/d\PB)\right]
    =J_\varphi(\nu|P),
  \end{align*}
  where note that
  $\frac{d\nu}{d\PB}/\frac{dP}{d\PB}=\frac{d\nu}{d\PB}$ if $\nu\ll P$
  etc. Thus it remains to show that
  \begin{align}
    \label{eq:RockafellarCbProof1}
    \sup_{f\in C_b(\Omega)}\left(\nu(f)-\EB_P[\varphi(\cdot, f)]\right)
    =\sup_{\xi\in L_\infty(P)}\left(\nu(\xi)-\EB_\PB[\varphi_P(\cdot, \xi)]\right).
  \end{align}
  Of course, ``$\leq$'' is clear. For ``$\geq$,'' let
  $\xi\in L_\infty(\PB)$ and pick a bounded representative $f\in\xi$ (relative
  to $L_\infty(\PB)$).  Now for each $\varepsilon>0$, Lusin's theorem
  yields a compact set $K_\varepsilon\subset\Omega$ such that
  $\PB(K_\varepsilon^c)<\varepsilon$ and $f|_{K_\varepsilon}$ is
  continuous, then Tietze's theorem gives us its continuous extension
  $f_\varepsilon\in C(\Omega)$ with
  $\|f_\varepsilon\|_\infty=
  \|f|_{K_\varepsilon}\|_\infty\leq\|\xi\|_\infty$. Then noting that
  $\|g\|_\infty\leq c$ $\Rightarrow$
  $|\varphi(\cdot, g)|\leq \varphi(\cdot,
  c)^++\varphi(\cdot,-c)^++c+\varphi^*(\cdot, \eta)^+=:\kappa_c\in
  L_1(P)$ where $\eta$ is as in (\ref{eq:IntegrandInteg2}),
  \begin{align*}
    \nu(\xi)
    &-\EB_\PB[\varphi_P(\cdot, \xi)]\\
    &= \nu(f_\varepsilon)-\EB_P[\varphi(\cdot, f_\varepsilon)]
      +\nu((f-f_\varepsilon)\ind_{K^c_\varepsilon})
      +\EB_P[\{\varphi(\cdot, f_\varepsilon)-\varphi(\cdot,f)\}\ind_{K^c_\varepsilon}]
    \\
    &\leq\sup_{g\in C_b(\Omega)}\left(\nu(g)-\EB_P[\varphi(\cdot,g)]\right)
      +2\|\xi\|_\infty|\nu|(K^c_\varepsilon)
      +2\EB_P[\kappa_{\|\xi\|_\infty}\ind_{K^c_\varepsilon}].
  \end{align*}
  The last two terms tend to $0$ as $\varepsilon\rightarrow 0$ since
  $P,|\nu|\ll\PB$.
\end{proof}

\section{Semistatic Robust Utility Indifference Valuation}
\label{sec:RobUtil}

We proceed to the robust utility maximisation problem. Let
$\Omega=\RB^N$ ($N\in\NB$), which we think of as the $N$-period
discrete time path-space, and $S=(S_i)_{1\leq i\leq N}$ the coordinate
process, i.e.
$(S_i(\omega))_{1\leq i\leq N}=\mathrm{id}_{\RB^N}(\omega)$,
$\omega\in\RB^N$, with $S_0(\omega)=s_0$ (constant) as the
(discounted) underlying assets.  Also, we are given a set $\PC$ of
Borel probability measures on $\RB^N$, viewed as the set of possible
models for $S$. As in Section~\ref{sec:RobInteg}, we suppose
\begin{equation}
  \label{eq:PAs1}
  \PC\text{ is a convex and $\sigma(\mathrm{ca}(\RB^N), C_b(\RB^N))$-compact}.
\end{equation}

Let $\HC$ denote the vector space of processes
$H=(H_i)_{1\leq i\leq N}$ such that
\begin{equation}
  \label{eq:SetH}
  H_1\text{ is constant; }H_t=h_t(S_1,...,S_{t-1})\text{ for some }h_t\in C_b(\RB^{t-1}),\, \forall i\geq 2.
\end{equation}
Each $H\in\HC$ is predictable (for the filtration generated by $S$),
and is thought of as a self-financing dynamic strategy with gain
$H\bullet S_t:=\sum_{i\leq t}H_i(S_i-S_{i-1})$, the discrete
stochastic integral. Note also that for a probability measure
$Q\in\mathrm{Prob}(\RB^N)$,
\begin{equation}
  \label{eq:MGMeasDiscrete1}
  S\text{ is a $Q$-martingale }\Leftrightarrow\, S_t\in L_1(Q),\,\forall t
  \text{ and }\EB_Q[H\bullet S_N]=0,\,\forall H\in\HC.
\end{equation}

The next ingredient is a family $\mu=(\mu_i)_{1\leq i\leq N}$ of
distributions on $\RB$ such that\nobreak
\begin{align*}
  \MC_\mu:=\left\{Q\in\mathrm{Prob}(\RB^N):
  S\text{ is a $Q$-martingale, }(Q\circ S_i^{-1})_{i\leq N}=(\mu_i)_{i\leq N}\right\}
  \neq\emptyset,
\end{align*}
for which it is necessary and sufficient that:
\begin{equation}
  \label{eq:Strassen1}
  \begin{minipage}{.85\linewidth}
    $\int|x|d\mu_i<\infty$, $\int x d\mu_i=s_0$ and
    $i\mapsto \int f d\mu_i$ is increasing for every convex function
    $f:\RB\rightarrow\RB$ (i.e. increasing in convex order).
  \end{minipage}
\end{equation}
This is \emph{\bfseries Strassen's theorem} (\citep{MR177430},
Th.~8). Then $\MC_\mu$ is a $\sigma(\ca(\RB^N),C_b(\RB^N))$-compact
convex set (\citep{MR3066985}, Prop.~2.4).  In (idealised) reality,
such a family $(\mu_i)_{i\leq N}$ is calculated from the prices of
call options via the relation (due to \citep{breeden78:_prices}):
\begin{align*}
  Q(\xi\leq K)
  =1+\lim_{\varepsilon\downarrow 0}\frac{\EB_Q[(\xi-K-\varepsilon)^+]-\EB_Q[(\xi-K)^+]}{\varepsilon},
\end{align*}
(if call options of all the strikes are available). In this sense,
each $Q\in\MC_\mu$ is a pricing measure \emph{\bfseries calibrated to
  the call prices in the market} (see e.g. \citep{MR2762363} for more
detailed exposition). Then every vanilla option with payoff function
$f\in C_b(\RB)$ and maturity $i\leq N$ is priced at
$\EB_Q[f(S_i)]=\mu_i(f)$ for all $Q\in\MC_\mu$; thus the final gain
from investing in $(f,i)$ is $f(S_i)-\mu_i(f)$. A static position is
any $(f_i)_{i\leq N}\in C_b(\RB)^N$ where each $f_i$ is a vanilla
option maturing at $i$, and any pair
$(H,(f_i)_{i\leq N})\in \HC\times C_b(\RB)^N$ is called a
\emph{\bfseries semistatic strategy}, whose gain is
\begin{align*}
  H\bullet S_N+\sum\nolimits_{i\leq N}(f_i(S_i)-\mu_i(f_i))
  =:H\bullet S_N+\Gamma_{(f_i)_{i\leq N}}.
\end{align*}

Finally, let $U:\RB\rightarrow\RB$ be a utility function (finite on
the whole $\RB$; e.g. exponential) that is strictly concave,
differentiable and satisfies the \emph{\bfseries Inada condition}
\begin{equation}
  \label{eq:Inada}
  \lim_{x\downarrow-\infty}  U'(x)=+\infty\quad\text{and}\quad
  \lim_{x\uparrow\infty}U'(x)=0.
\end{equation}
Then its conjugate $V(y):=\sup_{x\in\RB}(U(x)-xy)$ is a proper convex
function such that $\Int\dom(V)=(0,\infty)$ on which it is strictly
convex, differentiable, and
\begin{equation}
  \label{eq:VConj1}
V'(0):=\lim_{y\downarrow 0}V'(y)=-\infty,\, V'(\infty):=\lim_{y\uparrow\infty}V'(y)=+\infty.
\end{equation}
Now for each initial cost $x\in\RB$ and (the payoff function of) an
exotic option $\Psi:\RB^N\rightarrow\RB$, we consider the robust
utility maximisation with \emph{\bfseries semistatic strategies}:
\begin{equation}
  \label{eq:ValueFunctDiscrete1}
  u_\Psi(x):=
  \sup_{H\in\HC, f\in C_b(\RB)^N}\inf_{P\in\PC}\EB_P\left[U\left(x+H\bullet S_N +\Gamma_f-\Psi\right)\right].
\end{equation}

The main result of this Section is the following.
\begin{theorem}[Duality]
  \label{thm:DiscreteDuality}
  Suppose (\ref{eq:PAs1}), (\ref{eq:Strassen1}), (\ref{eq:Inada}) as
  well as
  \begin{align}
    \label{eq:UtilFinite1}
    &    \inf_{\lambda>0, Q\in\MC_\mu}J_{V,\PC}(\lambda Q)<\infty;\\
    \label{eq:DiscDualitySInteg1}
    &  
      \lim_n\sup_{P\in\PC}\EB_P\left[U(\alpha|S_i|)^-\ind_{\{|S_i|>n\}}\right]=0,
      \quad\forall i\in\{1,...,N\}, \forall \alpha>0.
  \end{align}
  Then for any upper semicontinuous function
  $\Psi:\RB^N\rightarrow\RB$ with linear growth (i.e.
  $|\Psi(\omega)|\leq c(1+|\omega_1|+\cdots+|\omega_N|)$ for some
  $c>0$), it holds that
  \begin{equation}
    \label{eq:DiscDuality1}
    u_\Psi(x)
    =\min_{\lambda> 0, Q\in\MC_\mu}\left(J_{V,\PC}(\lambda Q)
      -\lambda\EB_Q[\Psi]+\lambda x\right).
  \end{equation}
\end{theorem}

The duality easily gives us representations of associated robust
utility indifference prices as risk measures. Note that in view of
(\ref{eq:UtilFinite1}),
\begin{align*}
  \gamma_{V,\PC}(Q):=\inf_{\lambda>0}
  \frac1\lambda
  \left(J_{V,\PC}(\lambda Q)-u_0(0)\right),\quad Q\in \mathrm{Prob}\left(\RB^N\right),
\end{align*}
defines a positive proper convex function with
$\inf_{Q\in\mathrm{Prob}(\RB^N)}\gamma_{V,\PC}(Q)=0$.
\begin{corollary}[Indifference prices]
  \label{cor:IndiffPrice1}
  Under the assumptions of Theorem~\ref{thm:DiscreteDuality},
  \begin{enumerate}
  \item For any upper semicontinuous $\Psi:\RB^N\rightarrow\RB$ with
    linear growth,
    \begin{equation}
      \label{eq:IndiffSell}
      p_U^\mathrm{sell}(\Psi):=\inf\{x: u_\Psi(x)\geq u_0(0)\}
      =\sup_{Q\in\MC_\mu}\left(\EB_Q[\Psi]-\gamma_{V,\PC}(Q)\right).
    \end{equation}
    
  \item For any lower semicontinuous $\Psi:\RB^N\rightarrow\RB$ with linear growth,
    \begin{equation}
      \label{eq:IndiffBuy}
      p_U^\mathrm{buy}(\Psi)
      :=-p^{\mathrm{sell}}_U(-\Psi)
      =\inf_{Q\in\MC_\mu}\left(\EB_Q[\Psi]+\gamma_{V,\PC}(Q)\right).
    \end{equation}
  \item In particular, for any continuous $\Psi:\RB^N\rightarrow\RB$
    with linear growth,
    \begin{equation}
      \label{eq:IndiffSandwich}
      \inf_{Q\in\MC_\mu}\EB_Q[\Psi]
      \leq p^\mathrm{buy}_U(\Psi)\leq p^\mathrm{sell}_U(\Psi)
      \leq \sup_{Q\in\MC_\mu}\EB_Q[\Psi].
    \end{equation}

  \end{enumerate}
\end{corollary}

\begin{proof}
  (2) follows from (1), and (3) is a combination of (1) and (2).  The
  derivation of (1) from (\ref{eq:DiscDuality1}) is also standard: by
  (\ref{eq:DiscDuality1}), $u_\Psi(x)\geq u_0(0)$ iff for any
  $\lambda>0$ and $Q\in\MC_\mu$,
  $J_{V,\PC}(\lambda Q)-\lambda\EB_Q[\Psi]+\lambda x\geq u_0(0)$; then
  rearrange the terms and take the infimum over $\lambda>0$ and
  $Q\in\MC_\mu$.
\end{proof}

The estimate (\ref{eq:IndiffSandwich}) says that the indifference
prices lie in the \emph{\bfseries model-free pricing bound} in the
sense of \citep{MR3066985}:
\begin{equation}
  \label{eq:Bounds1}
  [p_U^\mathrm{buy}(\Psi),p_U^\mathrm{sell}(\Psi)]\subset
  \Bigl[\inf_{Q\in\MC_\mu}\EB_Q[\Psi],\sup_{Q\in\MC_\mu}\EB_Q[\Psi]\Bigr],
\end{equation}
for any continuous $\Psi:\RB^N\rightarrow \RB$ with linear growth
(then $\Psi\in\mcap_{Q\in\MC_\mu}L_1(Q)$).  In particular, the
indifference prices are consistent with the observed vanilla prices,
i.e.  if $\Psi(\omega)=g(\omega_i)=g(S_i(\omega))$ (i.e. a vanilla option with maturity $i$),
then
\begin{align*}
  p^{\mathrm{sell}}_U(\Psi)=p^\mathrm{buy}_U(\Psi)=\mu_i(g).
\end{align*}

\begin{remark}[A trivial case]
  \label{rem:Trivial}
  In the situation of the paper (with full marginal constraint),
  $\MC_\mu$ itself is (convex and)
  $\sigma(\ca(\RB^N),C_b(\RB^N))$-compact. If we take $\PC=\MC_\mu$,
  then $J_{V,\MC_\mu}(Q)\leq J_V(Q|Q)=V(1)<\infty$
  ($\forall Q\in\MC_\mu$), and $\gamma_{V,\MC_\mu}(Q)=0$ on $\MC_\mu$;
  thus in this case, the buyer's/seller's indifference prices
  coincide, respectively, with sub/super-hedging prices, i.e. the two
  intervals in (\ref{eq:Bounds1}) coincide. The choice of a ``nice''
  $\PC$ as well as a quantitative analysis are left for future topics.
\end{remark}

\begin{example}[Exponential case; cf. \citep{MR3910012}, \citep{guillaume:pastel-01002103}]
  \label{ex:Exponential}
  As one might expect, the situation is much simpler if the utility
  function is exponential, i.e.
  \begin{align*}
    U(x)=-e^{-x},\quad x\in\RB.
  \end{align*}
  Letting $\EC(Q|P)=\EB_P\bigl[\frac{dQ}{dP}\log\frac{dQ}{dP}\bigr]$
  if $Q\ll P$ and otherwise $+\infty$ (the relative entropy) and
  $\EC_\PC(Q)=\inf_{P\in\PC}\EC(Q|P)$, a straightforward calculation
  shows that
  \begin{align*}
    J_{V}(\lambda Q|P)
    =\lambda\EC(Q|P)+\lambda\log\lambda-\lambda.
  \end{align*}
  Thus for any continuous $\Psi:\RB^N\rightarrow\RB$ with linear
  growth and a $\sigma(\ca(\RB^N),C_b(\RB^N))$-compact convex set
  $\PC$ with $\inf_{Q\in\MC_\mu}\EC_\PC(Q)<\infty$ and
    \begin{equation}
    \label{eq:PAs2}
    \lim_k\sup_{P\in\PC}\EB_P[\exp(\alpha|S_i|)\ind_{\{|S_i|>k\}}]=0,\,\forall \alpha>0,\, i\leq N,
  \end{equation}
  one has
  \begin{align*}
    \sup_{H\in\HC,f\in C_b(\RB)^N}\inf_{P\in\PC}\EB_P\left[-e^{-(x+H\bullet S_N+\Gamma_f-\Psi)}\right]
    =-e^{-x- \min_{Q\in\MC_\mu}
    \left(\EC_\PC(Q)-\EB_Q[\Psi]\right)}.
  \end{align*}
  In particular,
  \begin{align*}
    p_{\exp}^{\mathrm{sell}}(\Psi)
    &=\max_{Q\in\MC_\mu}
      \Bigl\{\EB_Q[\Psi]
      -\Bigl(\EC_\PC(Q)
      -\inf_{Q'\in \MC_\mu}\EC(Q')\Bigr)\Bigr\}.
    \end{align*}
\end{example}

\subsection{Ramifications}
\label{sec:Ramifi}

From the financial motivation, it is important to note that the
duality (\ref{eq:DiscDuality1}) is somehow stable for the choice of
the admissible sets. In Theorem~\ref{thm:DiscreteDuality}, we chose
$C_b(\RB)^N$ for static positions and $\HC$ (given by (\ref{eq:SetH}))
for the dynamic ones. We first examine the largest choice. Let $\HC_s$
be the set of predictable processes $H$ (for the filtration generated
by $S$) such that $H\bullet S$ is a supermartingale under all
$Q\in\MC_\mu$, and consider $\prod_{i\leq N}L_1(\mu_i)$ for static
positions. Then for any
$f=(f_i)_{i\leq N}\in \prod_{i\leq N}L_1(\mu_i)$, $H\in\HC_s$ and
$Q\in\MC_\mu$,
\begin{align*}
  \EB_P[U(x+H\bullet S_N+\Gamma_f-\Psi)]
  &  \leq J_V(\lambda Q|P)+\lambda \EB_Q[x+H\bullet S_N+\Gamma_f-\Psi]\\
  \leq J_V(\lambda Q|P)+\lambda \EB_Q[x-\Psi],
\end{align*}
where
$\EB_Q[\Gamma_{f}] =\sum_{i\leq
  N}\left\{\EB_Q[f_i(S_i)]-\mu_i(f_i)\right\}=0$ if
$f\in \prod_{i\leq N}L_1(\mu_i)$; hence
\begin{align*}
  \sup_{H\in\HC_s,f\in\prod_{i\leq N} L_1(\mu_i)}
  & \inf_{P\in\PC}\EB_P[U(x+H\bullet S_N+\Gamma_f-\Psi)]\\
  &\leq \inf_{\lambda>0,Q\in\MC_\mu}
    \left(J_{V,\PC}(\lambda Q)+\lambda x-\lambda\EB_Q[\Psi]\right)
    =u_\Psi(x).
\end{align*}
Since $C_b(\RB)\subset L_1(\mu_i)$ and $\HC\subset\HC_s$, we deduce
that
\begin{corollary}
  \label{cor:DualityLargerAdm}
  Under the assumptions of Theorem~\ref{thm:DiscreteDuality}, it holds
  that
  \begin{align*}
    \sup_{H\in\HC_s,f\in\prod_{i\leq N}L_1(\mu_i)}  \inf_{P\in\PC}
    &\EB_P[U(x+H\bullet S_N+\Gamma_f-\Psi)]\\
    &=\min_{\lambda>0,Q\in\MC_\mu}
      \left(J_{V,\PC}(\lambda Q)+\lambda x-\lambda\EB_Q[\Psi]\right).
\end{align*}
\end{corollary}

A bit more general formulation is to choose, for each $i\leq N$, a
subset $\SC_i\subset L_1(\mu_i)$ for static positions maturing at $i$.
For instance, \citep{MR3910012} considered the case where each $\SC_i$
is spanned by a finite number (possibly $0$) of fixed options. Here we
consider the one spanned by call options of all the strikes:
\begin{align*}
  \SC_\mathrm{call}:=\Span\left((\cdot-K)^+: K\in\RB\right)\subset L_1(\mu_i).
\end{align*}
Note that every element of $\SC_\mathrm{call}$ is piecewise linear,
while any bounded piecewise linear function lies in
$\SC_\mathrm{call}+\RB=\{g+a:g\in\SC_\mathrm{call}, a\in\RB\}$.
\begin{corollary}[Duality with calls only]
  \label{cor:CallDuality}
  Under the assumptions of Theorem~\ref{thm:DiscreteDuality},
  \begin{equation}
    \label{eq:CallDuality}
    \begin{split}
      \sup_{H\in\HC,f\in\SC_\mathrm{call}^N}
      &\inf_{P\in\PC}\EB_P[U(x+H\bullet S_N+\Gamma_f-\Psi)]\\
      &=\min_{\lambda>0,Q\in\MC_\mu}\left(J_{V,\PC}(\lambda Q)+\lambda x-\lambda \EB_Q[\Psi]\right).
    \end{split}
  \end{equation}
\end{corollary}

\begin{proof}
  Will be given in Section~\ref{sec:ProofUtil}.
\end{proof}

\begin{remark}[Finitely many options]
  \label{rem:FinMany}
  Another possible (and rather more realistic) formulation is to set
  each $\SC_i$ to be the span of a finite number (possibly $0$) of
  fixed options, say $\SC_i=\Span(f_{i,1},...,f_{i,m_i})$ as
  considered in \citep{MR3910012} and \citep{MR4166747} (and
  \citep{MR3313756} in super-hedging). In this case, the duality
  (\ref{eq:CallDuality}) with \emph{\bfseries exact marginal
    constraint} is no longer true. But a similar duality with
  martingale measures $Q$ with constraints
  $\EB_Q[f_{i,k}(S_i)]=\mu_i(f_{i,k})$ holds; see \citep{MR3910012},
  Th.~2.2.
\end{remark}

\subsection{Proofs of Theorem~\ref{thm:DiscreteDuality} and
  Corollary~\ref{cor:CallDuality}}
\label{sec:ProofUtil}

We shall apply the results of Section~\ref{sec:RobInteg} to the normal
integrands
\begin{align*}
  \varphi_{H,\Psi}(\omega,x)=-U(-x+H\bullet S_N(\omega)-\Psi(\omega)),\quad H\in\HC.
\end{align*}
Note first that for each $H\in\HC$,
$\omega\mapsto H\bullet S_N(\omega)$ is continuous with linear growth,
i.e. $|H\bullet S_N(\omega)|\leq c(1+\sum_{i\leq N}|\omega_i|)$ for
some $c>0$ (say $c=2(|S_0|\vee 1)\max_{i\leq N}\|H_i\|_\infty$).  Thus
under the assumptions of Theorem~\ref{thm:DiscreteDuality},
$\omega\mapsto B_{H,\Psi}(\omega):=-H\bullet S_N(\omega)+\Psi(\omega)$
is usc with linear growth since $\Psi$ is. If $c>0$ is a linear growth
constant for $B_{H,\Psi}$, then letting $\varphi_U(x)=-U(-x)$, which
is convex,
\begin{align*}
  \varphi_U(\pm\alpha|B_{H,\Psi}|)
  &  \leq \varphi_U\Bigl(\alpha c\Bigl(1+\sum\nolimits_{i\leq N}|S_i|\Bigr)\Bigr)\\
  &\leq \frac{\varphi_U(2\alpha c)^+}{2}+\sum\nolimits_{i\leq N}\frac{\varphi_U(2\alpha cN|S_i|)^+}{2N}
    \stackrel{\text{(\ref{eq:DiscDualitySInteg1})}}\in\LB_{1,b}(\PC),
\end{align*}
for any $\alpha>0$.  Hence Proposition~\ref{prop:GoodIntegrand}
applied to $\varphi_U$ and $B_{H,\Psi}$ yields that
$\varphi_{H,\Psi}(\cdot, x)=(\varphi_U)_{B_H}(\cdot,x)$ satisfies
Assumption~\ref{as:Integrand}, and we have

\begin{lemma}
  \label{lem:FinInteg1}
  Under the assumptions of Theorem~\ref{thm:DiscreteDuality}, it holds
  that for any $H\in\HC$,
  \begin{align*}
    I_{\varphi_{H,\Psi},\PC}(f)=-\inf_{P\in\PC}\EB_P\left[U(-f+H\bullet
      S_N-\Psi)\right]
  \end{align*}
  is continuous on $C_b(\RB^N)$ for $\tau(C_b(\RB^N),\ca(\RB^N))$;
  $-H\bullet S_N+\Psi\in L_1(\nu)$ whenever
  $\inf_{\lambda>0}J_{V,\PC}(\lambda\nu)<\infty$; and the conjugate of
  $I_{\varphi_{H,\Psi},\PC}$ is given on $\ca(\RB^N)$ as
        \begin{equation}
    \label{eq:UtilIntegConj1}
    I_{\varphi_{H,\Psi},\PC}^*(\nu)=
    \ccases{
      J_{V,\PC}(\nu)+\nu(H\bullet S_N-\Psi)&\text{if }\inf_{\alpha>0}J_{V,\PC}(\alpha\nu)<\infty,\\
      +\infty &\text{otherwise.}}
  \end{equation}
  In particular,
  $\dom(I^*_{\varphi_{H,\Psi},\PC})=\dom(J_{V,\PC})\subset\ca(\RB^N)^+$
  (since $\dom(V)\subset\RB^+$), and $\Psi,S_i\in L_1(\nu)$,
  $i\leq N$, whenever $\inf_{\lambda>0}J_{V,\PC}(\lambda\nu)<\infty$.
\end{lemma}

Recall that $\Gamma_f=\sum_{i\leq N}(f_i(S_i)-\mu_i(f_i))$ for
$f=(f_i)_{i\leq N}\in C_b(\RB)^N$, and let
\begin{align}
  \label{eq:D}
  \DC:=\left\{\Gamma_f: f\in C_b(\RB)^N\right\}.
\end{align}
This is a vector subspace of $C_b(\RB^N)$, so its (one-sided) polar in
$\ca(\RB^N)$ is
$\DC^\medcirc=\{\nu\in\ca(\RB^N): \nu(\psi)=0, \psi\in\DC\}$ (which is
linear), and a probability measure $Q$ lies in $\DC^\medcirc$ iff
$\EB_Q[f(S_i)-\mu_i(f)]=0$ for $i\leq N$ and $f\in C_b(\RB)$ iff
$Q\circ S_i^{-1}=\mu_i$, $i\leq N$.  In other words,
\begin{equation}
  \label{eq:ProofUtilMarginal1}
  \DC^\medcirc\cap\ca(\RB^N)^+
  =\left\{\lambda Q:\lambda\geq 0, Q\in\mathrm{Prob}(\RB^N), Q\circ S_i^{-1}=\mu_i, i\leq N\right\}.
\end{equation}
In particular, by (\ref{eq:UtilFinite1}) and
$\dom(J_{V,\PC})\subset\ca(\RB^N)$ (Lemma~\ref{lem:FinInteg1}),
\begin{equation}
  \label{eq:ProofUtilNonEmpty1}
  \emptyset\neq
  \MC_\mu\cap\dom(J_{V,\PC})\subset\DC^\medcirc\cap\dom(J_{V,\PC})\subset\ca(\RB^N)^+,
\end{equation}
Note also that in view of (\ref{eq:Strassen1}) and the linear growth
assumption on $\Psi$,
\begin{equation}
  \label{eq:ProofUtilInteg1}
  \nu\in\DC^\medcirc\cap\ca(\RB^N)^+\,\Rightarrow S_i\in L_1(\nu),\text{ so }\Psi, H\bullet S_N\in L_1(\nu),
  \,\forall H\in\HC.
\end{equation}
Consequently, on $\DC^\medcirc\ca(\RB^N)^+$, (\ref{eq:UtilIntegConj1})
simplifies to
\begin{equation}
  \label{eq:UtilConjBetter1}
  I^*_{\varphi_{H,\Psi},\PC}(\nu)=J_{V,\PC}(\nu)+\nu(H\bullet S_N-\Psi),\quad \nu\in\DC^\medcirc\cap\ca(\RB^N)^+.
\end{equation}

\begin{proof}[Proof of Theorem~\ref{thm:DiscreteDuality}]
  Replacing $\Psi$ by $\Psi-x$, which does not affect the assumptions
  on $\Psi$, it suffices to prove the case of $x=0$. In this case,
  Corollary~\ref{cor:MetaDuality} and Lemma~\ref{lem:FinInteg1} show
  that
  \begin{align*}
    \sup_{f\in\DC} \inf_{P\in\PC}\EB_P[U(f+H\bullet S_N-\Psi)]
    =-\inf_{f\in -\DC}I_{\varphi_{H,\Psi},\PC}(f)
    =\min_{\nu\in\DC^\medcirc}I^*_{\varphi_{H,\Psi},\PC}(\nu).
  \end{align*}
  Since $\dom(I_{\varphi_{H,\Psi},\PC})=\dom(J_{V,\PC})$, we deduce
  that
  \begin{align*}
    \min_{\nu\in\DC^\medcirc}I^*_{\varphi_{H,\Psi},\PC}(\nu)
    =\min_{\nu\in\DC^\medcirc\cap\dom(J_{V,\PC})}I^*_{\varphi_{H,\Psi},\PC}(\nu)
    \stackrel{\text{(\ref{eq:ProofUtilNonEmpty1})}}<\infty.
  \end{align*}
  Then note that
  $\{\nu\in \DC^\medcirc\cap\dom(J_{V,\PC}):
  I_{\varphi_{H,\Psi},\PC}^*(\nu)\leq c\} =\DC^\medcirc\cap\{\nu\in
  \ca(\RB^N): I_{\varphi_{H,\Psi},\PC}^*(\nu)\leq c\}$, $c\in\RB$, are
  convex and $\sigma(\ca(\RB^N),C_b(\RB^N))$-compact since
  $I_{\varphi_{H,\Psi},\PC}$ is
  $\tau(C_b(\RB^N),\ca(\RB^N))$-continuous, and
  $H\mapsto I^*_{\varphi_{H,\Psi},\PC}(\nu)$ is concave for each
  $\nu\in \DC^\medcirc\cap\dom(J_{V,\PC})$. Thus the lop-sided minimax
  theorem (\citep[Th.6.2.7 on p.319]{MR749753}) shows
  \begin{align*}
      \sup_{H\in\HC}
      &\min_{\nu\in\DC^\medcirc\cap\dom(J_{V,\PC})}I^*_{\varphi_{H,\Psi},\PC}(\nu)
        = \min_{\nu\in\DC^\medcirc\cap\dom(J_{V,\PC})}\sup_{H\in\HC}I^*_{\varphi_{H,\Psi},\PC}(\nu)\\
      &\stackrel{\text{(\ref{eq:UtilConjBetter1})}}=
        \min_{\nu\in\DC^\medcirc\cap\dom(J_{V,\PC})}
        \left(J_{V,\PC}(\nu)+\sup_{H\in\HC}\nu(H\bullet S_N)
        -\nu(\Psi)\right)
        =:(*).
  \end{align*}
  Then note that for $\nu\in\DC^\medcirc\cap\dom(J_{V,\PC})$,
  $\sup_{H\in\HC}\nu(H\bullet S_N)\in\{0,\infty\}$ since $\HC$ is
  linear (and $S_i\in L_1(\nu)$ by (\ref{eq:ProofUtilInteg1})), and by
  (\ref{eq:MGMeasDiscrete1}), it is $0$ iff $\nu=\alpha Q$ for a
  martingale measure $Q$ for $S$ and $\alpha\geq 0$. Summing up with
  (\ref{eq:ProofUtilMarginal1}) (and
  $\dom(J_{V,\PC})\subset\ca(\RB^N)^+$),
  \begin{align*}
    (*)&=\min_{\lambda\geq 0,Q\in\MC_\mu}\left(J_{V,\PC}(\lambda Q)-\lambda \EB_Q[\Psi]\right).
  \end{align*}
  
  We complete the proof by showing that the minimum on the RHS is
  attained by a non-zero $\lambda Q$. To see this, note first that
  $J_{V,\PC}(0)-0(\Psi)=V(0)$. Pick, by (\ref{eq:UtilFinite1}), a
  $Q\in\MC_\mu$ with $J_V(\lambda Q|P)<\infty$ for some $\lambda>0$
  and $P\in\PC$; then $Q\ll P$, so $P(dQ/dP>0)>0$.  Putting
  $\eta:=\lambda dQ/dP$,
  $\alpha \mapsto G(\alpha):=V(\alpha\eta)-\alpha\eta \Psi$ is
  (finite-valued and) convex. Since
  $V'(0)=\lim_{\downarrow 0}V'(\alpha)=-\infty$,
  $\frac{G(\alpha)-G(0)}\alpha\downarrow -\infty\ind_{\{\eta>0\}}$,
  and since $G(1)=V(\eta)-\eta\Psi\in L_1(P)$, we deduce that
  \begin{align*}
    \frac{  J_V(\alpha\nu|P)-\alpha\nu(\Psi)-V(0)}\alpha
    =
    \EB_P\left[\frac{G(\alpha)-G(0)}\alpha\right]\downarrow -\infty.
  \end{align*}    
  Thus $J_V(\alpha \nu|P)-\alpha \nu(\Psi)<V(0)$ for some
  $\alpha>0$.
\end{proof}

For the proof of Corollary~\ref{cor:CallDuality}, we need a simple lemma. Let
\begin{align*}
  \DC_\mathrm{call}
  =\left\{\sum\nolimits_{i\leq N} (f_i(S_i)-\mu(f_i)): f_i\in \SC_\mathrm{call}\right\}.
\end{align*}

\begin{lemma}
  \label{lem:MackeyApprox}
  Any $\psi\in \DC$ admits a sequence $(\psi_n)_n$ in
  $\DC_\mathrm{call}\cap C_b(\RB)$ with $\psi_n\rightarrow\psi$ in
  $\tau(C_b(\RB^N),\ca(\RB^N))$.
\end{lemma}

\begin{proof}
  Since $\DC$ consists of functions
  $\sum_{i\leq N}(f_i\circ S_i-\mu_i(f_i))$ with $f_i\in C_b(\RB)$,
  and $(f_i+a)\circ S_i-\mu_i(f_i+a)=f_i\circ S_i-\mu_i(f_i)$ if $a$
  is a constant, it suffices to show that each $g\in C_b(\RB)$ admits
  a sequence $(g_n)_n$ of bounded piecewise linear functions on $\RB$
  such that $g_n\circ S_i-\mu_i(g_n)\rightarrow g\circ S_i-\mu_i(g)$
  in $\tau(C_b(\RB^N),\ca(\RB^N))$. So fix $g\in C_b(\RB)$. For each
  $n$, $g$ is uniformly continuous on $[-n,n]$, so one can find a
  piecewise linear function $g_n:[-n,n]\rightarrow\RB$ with
  $|g-g_n|\leq 1/n$ on $[-n,n]$. Extend $g_n$ to the entire $\RB$ by
  setting $g(x)=g(-n)$ (resp. $g(n)$) if $x<-n$ (resp. $>n$), which is
  piecewise linear and $|g_n|\leq 2\|g\|_\infty$ on
  $\RB\setminus [-n,n]$. Now for each $i\leq N$,
  \begin{align*}
    |\mu_i(g-g_n)|\leq\frac1n+2\|g\|_\infty\mu_i\left(\RB\setminus[-n,n]\right)\rightarrow 0.
  \end{align*}
  Also, for any $\sigma(\ca(\RB^N),C_b(\RB^N))$-compact
  ($\Leftrightarrow$ bounded uniformly tight)
  $\Lambda\subset \ca(\RB^N)$,
  \begin{align*}
    \sup_{\nu\in\Lambda}|\nu(g\circ S_i-g_n\circ S_i)|
    \leq \frac1n\sup_{\nu\in \Lambda}\underbrace{|\nu|\left([-n,n]^N\right)}_{\leq\|\nu\|}
    + 2\|g\|_\infty\sup_{\nu\in\Lambda}|\nu|\left(\RB^N\setminus[-n,n]^N\right).
  \end{align*}
  By the uniform tightness, the RHS tends to $0$ as
  $n\rightarrow\infty$. This shows that
  $g_n\circ S_i\rightarrow g\circ S_i$ in
  $\tau(C_b(\RB^N),\ca(\RB^N))$. Summing up,
  $g_n\circ S_i-\mu_i(g_n)\rightarrow g\circ S_i-\mu_i(g)$ in
  $\tau(C_b(\RB^N),\ca(\RB^N))$.
\end{proof}

\begin{proof}[Proof of Corollary~\ref{cor:CallDuality}]
  Since
  $f\mapsto \inf_{P\in\PC}\EB_P[U(f+H\bullet
  S_N-\Psi)]=-I_{\varphi_{H,\Psi},\PC}(-f)$ is
  $\tau(C_b(\RB^N),\ca(\RB^N))$-continuous on $C_b(\RB^N)$ for each
  $H\in\HC$, Lemma~\ref{lem:MackeyApprox} yields
\begin{align*}
u_\Psi(0)=  \sup_{H\in\HC}\sup_{\psi\in \DC}-I_{\varphi_{H,\Psi},\PC}(-\psi)
  &\leq\sup_{H\in\HC}\sup_{\psi\in \DC_\mathrm{call}}-I_{\varphi_{H,\Psi},\PC}(-\psi)\\
  &\leq\min_{\lambda>0,Q\in\MC_\mu}\left(J_{V,\PC}(\lambda Q)-\lambda\EB_Q[\Psi]\right)
    =u_\Psi(0),
\end{align*}
where in the second line, note that any $g\in\SC_\mathrm{call}$ (is
Lipschitz, hence) has a linear growth, so $g\in L_1(\mu_i)$, thus
Corollary~\ref{cor:DualityLargerAdm} proves the second ``$\leq$.''
\end{proof}

\def\cprime{$'$}

\end{document}